\documentclass[11pt]{article}
\usepackage[utf8]{inputenc}
\usepackage[T1]{fontenc}
\usepackage{amsmath,amssymb,amsthm}
\usepackage[margin=1in]{geometry}
\usepackage{xcolor}
\usepackage{hyperref}
\hypersetup{colorlinks=true,linkcolor=blue,citecolor=blue,urlcolor=blue}

\newtheorem{theorem}{Theorem}
\newtheorem{lemma}[theorem]{Lemma}
\newtheorem{proposition}[theorem]{Proposition}
\newtheorem{corollary}[theorem]{Corollary}
\newtheorem{definition}[theorem]{Definition}
\theoremstyle{remark}
\newtheorem{remark}[theorem]{Remark}

\title{Natural Geometry Based on the Rectilinear Plane}
\author{Mohammed Berkani\\[3pt]
\small Mohammed~I University, Science Faculty\\
\small Honorary member, Mathematics Laboratory Ibn Al-Banna\\
\small  Oujda, Morocco}
\date{\today}

\begin{document}
\maketitle

\begin{flushright}
\itshape Geometry is the visualisation of our intuitive understanding of the universe.
\end{flushright}

\begin{abstract}
The aim of this note is to introduce natural geometry, in which the rectilinear plane is a
primitive. This plane is admitted together with its rectitude, and its rectitude states that a
Lambert quadrilateral [\ref{Lambert}] has its fourth angle right. The fourth angle of a Lambert
quadrilateral may be right, acute or obtuse, and this trichotomy distinguishes three planes.
Starting from the rectitude, we prove the uniqueness of the parallel by a chain of neutral
steps: the rectangle, the triangle whose angles sum to $\pi$, and Legendre's
theorem \cite{legendre}, all borrowed from the classical literature. That uniqueness is equivalent
to the fifth postulate, which we thus reach in Euclid's own wording: the equivalence is read in the
natural direction, and it is not a derivation of the fifth postulate from the other four, since the
independence of that postulate is not in question. In the rectilinear plane the parallels are
characterised as exactly the images of a line under a point-reflection, and equivalently as the
pairs of equidistant lines. The other two planes are then placed beside the rectilinear one,
and the same definitions serve for all three.
\end{abstract}

\section{Introduction}

Euclid's \emph{Elements} (c.~300~BC) deduces the elementary geometry of the plane from five
postulates \cite{heath1,heath2,heath3}:

\begin{enumerate}
\item To draw a straight line from any point to any point.
\item To produce a finite straight line continuously in a straight line.
\item To describe a circle with any centre and distance.
\item That all right angles are equal to one another.
\item That, if a straight line falling on two straight lines makes the interior angles on the
same side less than two right angles, the two straight lines, if produced indefinitely, meet on
that side on which the angles are less than two right angles.
\end{enumerate}

The first four postulates are immediately acceptable. The fifth did not satisfy Euclid himself,
who deferred its use for as long as he could: the first twenty-eight propositions of Book~I \cite{heath1} are
proved without it, and it is first invoked at \cite[Proposition 29]{heath1}. 

From the Greek period onwards, the question
was raised whether the fifth postulate might be a consequence of the other four. Proclus
(410--485) reformulated it and offered a proof of his own, but that proof is fallacious. The
Arabic tradition took up the question with remarkable care. \emph{Omar Khayyam} (1048--1131), in
his commentary on the difficulties of Euclid's postulates, studied the quadrilateral with two
equal perpendiculars erected on a base, and he examined the three hypotheses on its summit
angles; this figure is now called the \emph{Khayyam--Saccheri quadrilateral}
\cite{rosenfeld,khayyam}. \emph{Nasir al-Din al-Tusi} (1201--1274) returned to the question in the
treatise that relieves the doubt concerning parallel lines, written not later than 1251; Rosenfeld
places it in the same tradition \cite{rosenfeld,tusi}. Saccheri (1733) \cite{saccheri1733}
attempted a refutation by \emph{reductio ad absurdum}, and he came within a hair's breadth of a
geometry in which the postulate fails. Lambert (1766) pursued the acute-angle case, but he
obtained no contradiction.
Legendre returned to the question repeatedly throughout his \emph{Éléments de géométrie}
\cite{legendre}. The question was settled in the negative by Lobachevsky (1829) and,
independently, by J\'anos Bolyai (1832), who exhibited a consistent geometry in which the fifth
postulate fails. For the history of these investigations we refer to Greenberg \cite{greenberg}
and Rosenfeld \cite{rosenfeld}. 

Hilbert's \emph{Grundlagen der Geometrie} (1899) \cite{hilbert} settled the foundational
question by making the primitives \emph{neutral}, so that point, line and plane are undefined and
the parallel property
becomes an added axiom, whose independence is established by models
\cite{beltrami1868,greenberg}. In that system the same word denotes different objects as soon as
the parallel axiom is changed, and the several geometries are, in form, reinterpretations of a
single schema.

The present note proceeds otherwise, and it claims less. We take as our sole primitive the
\emph{rectilinear plane}, that is, the plane of ordinary intuition, admitted rather than
constructed. The straight line is then defined once and for all as a shortest path, and parallel
lines as non-meeting ones. Under these definitions the existence of the parallel
is a theorem of neutral geometry, while its uniqueness follows from the primitive. Euclid's fifth
postulate is therefore neither derived, for it is not derivable, nor added as an axiom beside the
others: it is \emph{admitted} as the rectitude of this plane, and it is then \emph{recovered} in
Euclid's own wording, in the precise sense that Section~\ref{sec:primitive} makes explicit. Under
the same two definitions, the other classical geometries appear as the same construction read on
particular spaces, and those spaces differ by their curvature.

The three positions may be summarised as follows.     Euclid took point, line and plane as basic 
notions (giving only informal definitions of them)
 and added the parallel property as a fifth postulate; that postulate is
therefore undemonstrable, and its independence was established only in the nineteenth century.
Hilbert took point, line and plane as undefined primitives, and he kept the parallel property as
an added axiom, whose independence he proved by models. Natural geometry takes a single
primitive, the rectilinear plane, and it admits the parallel property as the \emph{rectitude} of
that plane; the uniqueness of the parallel follows from it, and the fifth postulate is recovered
in Euclid's wording. The third position does not escape the statement; it locates the statement.
What Euclid and Hilbert postulate or prove, natural geometry states once, as the sense of its
primitive.

\section{The primitive and the definitions}
\label{sec:primitive}
\begin{definition}[The rectilinear plane]
The \emph{rectilinear plane} is the plane of ordinary intuition, admitted and not constructed, 
together with its points and its distance. 
Its \emph{rectitude} is the property that every Lambert quadrilateral is a rectangle 
(that is, its fourth angle is also right).
\end{definition}

\begin{definition}[Lambert quadrilateral]\label{Lambert}
A \emph{Lambert quadrilateral} is a quadrilateral that has three right angles.
\end{definition}

\begin{remark}
Such quadrilaterals were already considered by Ibn al-Haytham in the early eleventh century
\cite[p.~160]{greenberg}. For this reason they may be called \emph{Ibn al-Haytham--Lambert
quadrilaterals}.
\end{remark}

\begin{definition}[Segment]
The \emph{segment} joining two distinct points $P$ and $Q$ is a shortest path from $P$ to $Q$.
\end{definition}

The indefinite article is deliberate: in the rectilinear plane a shortest path between two
points is unique, whereas on the sphere, between antipodal points, there are many.

It is well known that Euclid's original definition of a straight line—as a line which lies evenly with the points on itself—is one of the primary 
sources of confusion and conceptual ambiguity in his axiomatic system. 
To resolve this lack of precision in  Euclid's definition, the following formalization is adopted.

\begin{definition}[Straight line]\label{def:straight_line}
A straight line is a set of points, maximal with respect to inclusion, such that for any two of its points the part of the set between them is a segment.
\end{definition}

This is a single definition, valid in every geometry; it names what geometers call a geodesic.
What varies from one space to another is the distance, not the wording.

\begin{definition}[Point-reflection]
The \emph{point-reflection} in a point $O$, written $\sigma_O$, sends $A$ to $A'$ when $O$ is
the midpoint of the segment $AA'$. It is an involution, $\sigma_O\circ\sigma_O=\mathrm{id}$, the identity.
\end{definition}

\begin{definition}[Parallelism]\label{def:parallel}
Two coplanar lines are \emph{parallel} when, produced indefinitely in both directions, they do
not meet.
\end{definition}

\begin{definition}[Natural geometry]
\emph{Natural geometry} is the way of visualising the universe. It presents three planes, the
rectilinear, the spherical and the hyperbolic. The three planes differ by their curvature, which is zero (flat), positive, or negative respectively. The spherical case has a second form, the elliptic plane, in which antipodal points are identified; those four spaces are the particular cases of this note. Among them, the rectilinear plane
is distinguished by its rectitude.
\end{definition}

\section{Existence and uniqueness in the rectilinear plane}
\label{sec:uniq}

Existence uses the point-reflection and the foot of a perpendicular, both of which the
rectilinear plane supplies, and uniqueness turns on one further fact. That fact does not
follow from Definitions~\ref{def:straight_line} and~\ref{def:parallel} alone: those two
definitions hold verbatim for the geodesics of the hyperbolic plane, where the foot and the
point-reflection exist as well and the existence argument goes through, and yet uniqueness fails;
it is the rectitude of the plane, that is, the primitive, which decides.

\begin{proposition}[Existence]\label{prop:existence}
Through a point exterior to a line there passes a parallel to it.
\end{proposition}

\begin{proof}
A point-reflection carries a line to a line; this neutral fact, like the others used here, is
borrowed. Let $Q$ be exterior to $\ell$, let $H$ be the foot of the perpendicular from $Q$ to
$\ell$, and let $O$ be the midpoint of $QH$. Then $\sigma_O(H)=Q$, so $\sigma_O(\ell)$ passes
through $Q$. Moreover $\sigma_O(\ell)$ does not meet $\ell$: were the two lines to meet at a
point $X$, then $\sigma_O(X)$ would lie on both of them, the segment joining $X$ to
$\sigma_O(X)$ would lie on $\ell$, and its midpoint $O$ would lie on $\ell$; hence
$\sigma_O(\ell)=\ell$, contrary to $\sigma_O(H)=Q$. Thus $\sigma_O(\ell)$ is parallel to $\ell$
in the sense of Definition~\ref{def:parallel}.
\end{proof}

\begin{definition}[Khayyam--Saccheri quadrilateral]\label{def:khayyam_saccheri}
A \emph{Khayyam--Saccheri quadrilateral} is a quadrilateral \( ABCD \) in which the base angles are right and the legs are equal:
\[
AD \perp AB,\qquad BC \perp AB,\qquad AD = BC.
\]
The side \( AB \) is called the \emph{base}, the side \( CD \) the \emph{summit}, and the angles at \( C \) and \( D \) the \emph{summit angles}.
\end{definition}

\begin{remark}\label{rem:quadrilateral}

Let \( M \) and \( N \) be the midpoints of the base \( AB \) and of the summit \( CD \) respectively. 
Then the segment \( MN \) is perpendicular to both \( AB \) and \( CD \) (see \cite[Lemma 6.2]{greenberg}). 
Consequently each of the two halves \( AMND \) and \( MBCN \) is a Lambert quadrilateral.
\end{remark}

\begin{lemma}[Rectitude]\label{lem:rectitude}
In the rectilinear plane a rectangle exists, and there is a triangle whose angles sum to $\pi$.
\end{lemma}

\begin{proof}
Take a line $\ell$, two distinct points $H$ and $H_1$ on it, and on the same side of $\ell$ the
perpendiculars $HQ$ and $H_1Q_1$ of equal length. The quadrilateral $HQQ_1H_1$ has right angles at
$H$ and at $H_1$, and its legs $HQ$ and $H_1Q_1$ are equal; it is therefore a Khayyam--Saccheri 
quadrilateral. In it the summit angles at $Q$ and at $Q_1$ are equal, and the line joining the
midpoint $N$ of $HH_1$ to the midpoint $M$ of $QQ_1$ is perpendicular to both, by the previous remark. The quadrilateral
$HNMQ$ therefore is a Lambert quadrilateral . By the
rectitude of the plane its fourth angle, at $Q$, is right, and so a rectangle exists. We cut it
along a diagonal. The diagonal partitions two corners of the rectangle and leaves the other two
unsplit, so the six angles of the two triangles recombine to the four right angles and sum to
$2\pi$. By Legendre's theorem \cite[Book~I, Prop.~XIX]{legendre}, which uses the Archimedean
hypothesis \cite[Thm.~4.4 and p.~125]{greenberg}, each sum is at most $\pi$, and hence each equals $\pi$. Thus a
triangle of angle sum $\pi$ exists. 
\end{proof}

\begin{theorem}[Uniqueness]\label{thm:uniqueness}
In the rectilinear plane, through a point exterior to a line there passes exactly one parallel
to that line.
\end{theorem}

\begin{proof}
Let $Q$ be exterior to $\ell$. By Lemma~\ref{lem:rectitude} there is a triangle whose angles
sum to $\pi$, and by Legendre's theorem \cite[Book~I, Props.~XX and~XXIV]{legendre} the existence of
such a triangle is equivalent to the uniqueness of the parallel \cite[Theorem~4.7]{greenberg}. Hence at most one line through
$Q$ is disjoint from $\ell$, and by Proposition~\ref{prop:existence} there is exactly one.
\end{proof}

\begin{proposition}[Equidistance of parallels]\label{prop:equidistance}
In the rectilinear plane, two parallel lines are equidistant: the perpendicular distance from a
point of one to the other is the same for every such point.
\end{proposition}

\begin{proof}
Let $m$ and $\ell$ be parallel, let $Q$ and $Q'$ be two distinct points of $m$, and let $H$ and
$H'$ be the feet of the perpendiculars from $Q$ and from $Q'$ to $\ell$.As $m$ and $\ell$ are parallel they are distinct (a line meets itself), so $m$ lies in one of the open half-planes bounded by $\ell$; hence $Q$ and $Q'$ are on the same side, and $H\neq H'$.  The lines $QH$ and
$Q'H'$ are both perpendicular to $\ell$, and two lines perpendicular to the same line are
parallel; since $QQ'\subset m$ is parallel to $\ell\supset HH'$, the quadrilateral $QHH'Q'$ has
its two pairs of opposite sides parallel, that is, it is a parallelogram.  By Theorem~\ref{thm:uniqueness} the parallel through a point is unique, so the classical theory of parallels applies \cite[Book~I, Props.~29 and~34]{heath1}: in a parallelogram the opposite angles are equal. 
Since the angles of $QHH'Q'$ at $H$ and at $H'$ are right, then  $\angle Q=\angle Q'= \pi/2$ and $QHH'Q'$ is
a rectangle. In a rectangle the opposite sides are equal, so $QH=Q'H'$. $QH$ is the distance from $Q$ to $\ell$, since the perpendicular is the shortest segment from a point to a line.  The point $Q'$ is an
arbitrary point of $m$, and the distance is therefore the same for every such point.
\end{proof}

\begin{proposition}[Parallels as point-reflection images]\label{prop:reflection_parallel}
In the rectilinear plane, let $\ell$ and $m$ be two distinct lines. Then $\ell$ and $m$ are
parallel if and only if $m$ is a point-reflection image of $\ell$.
\end{proposition}

\begin{proof}
If $\sigma_O(\ell)$ meets $\ell$ at a point $X$, then $\sigma_O(X)$ also lies on $\ell$, so the
segment joining $X$ to $\sigma_O(X)$ lies on $\ell$, its midpoint $O$ lies on $\ell$, and
$\sigma_O(\ell)=\ell$. Hence an image that does not coincide with $\ell$ does not meet $\ell$: it
is a line parallel to $\ell$.
If $m=\sigma_O(\ell)$ for some point $O$, then $O$ does not lie on $\ell$, for otherwise
$\sigma_O(\ell)=\ell$, contrary to $m\neq\ell$; by the previous paragraph the lines $\ell$ and $m$
do not meet, that is, they are parallel.
Conversely, let $\ell$ and $m$ be parallel, let $Q$ be a point of $m$, let $H$ be the foot of the
perpendicular from $Q$ to $\ell$, and let $O$ be the midpoint of $QH$. Then $\sigma_O(H)=Q$ lies on
the image $\sigma_O(\ell)$, which is a line through $Q$; since $O$ does not lie on $\ell$ (it is an
interior point of the segment $QH$, whose only point on $\ell$ is $H$), that image is parallel to
$\ell$ by the first paragraph; and by Theorem~\ref{thm:uniqueness} the parallel through $Q$ is
unique, so $\sigma_O(\ell)=m$.
\end{proof}

\begin{corollary}\label{cor:equidistance}
In the rectilinear plane, let $\ell$ and $m$ be two distinct lines. Then $\ell$ and $m$ are
parallel if and only if they are equidistant.
\end{corollary}

\begin{proof}
If the lines are parallel, they are equidistant by Proposition~\ref{prop:equidistance}. Conversely,
if they are equidistant they cannot meet, so they are parallel.
\end{proof}

\begin{remark}[Admissible centres]
The centre of Proposition~\ref{prop:existence} is the midpoint of $QH$. Any point of the
midline between $\ell$ and the parallel through $Q$ is also a centre, and all of them give the
same image. This agreement is peculiar to the rectilinear plane: on a curved plane the
admissible centres fan out, and the images diverge. The reason is that the point-reflection is an
isometry only where the curvature is the same at every point, and it is that symmetry which makes
the construction of Proposition~\ref{prop:existence} canonical.
\end{remark}

\section{Particular spaces}
\label{sec:spaces}

\, \, \,\, \textbf{Hyperbolic plane}: the lines are the geodesics, that is, the maximal sets whose part
between two points is a segment (Definition~\ref{def:straight_line}). Through a point exterior
to $\ell$ there
pass \emph{several} lines which do not meet $\ell$: the \emph{ultraparallels}, each with its own
common perpendicular, and the \emph{limiting parallels} of Lobachevsky \cite[Ch.~6, p.~183]{greenberg}, which approach $\ell$
asymptotically. All of them are parallel in the sense of Definition~\ref{def:parallel}, and they
are not unique: in this plane the uniqueness of the parallel fails. Among them, the
point-reflection images of $\ell$ are the ultraparallels alone. They are not equidistant either:
the distance is minimal along the common perpendicular and increases away from it.

\textbf{Spherical plane:} the lines are the great circles. Any two distinct great circles meet in a pair of
antipodal points. The point-reflection of a great circle is a great circle, and the two meet. No
pair of lines satisfies Definition~\ref{def:parallel}. Arcs that are point-reflections of one
another do exist; extended to maximal lines, they meet. The parallel relation is empty. The
sphere is compact, and it is not another copy of the rectilinear plane.

\textbf{Elliptic plane:} the sphere with antipodal points identified. Lines meet in exactly one point,
the image of a line under a point-reflection is a line, and it meets the original. No parallel
exists. The case is recorded here so that this absence is not confused with a failure of
Definition \ref{def:straight_line}.

\begin{remark}[Intuition and variable curvature]\label{rem:outlook}
The primitive is not a postulate but an intuition, and an intuition, once visualised, can be
enriched. The difference with Hilbert lies there and not in the axioms: his primitives are names
of deliberately empty sense, and he said so himself: ``One must be able to say at all times---
instead of points, lines and planes---tables, chairs and beer mugs'' \cite[p.~72]{greenberg}. The
rectilinear plane, by contrast, is an intuition that everyone imagines; it is one of the three
planes of \emph{constant} curvature, zero (flat), positive or negative, and it is constant curvature
that makes each of them uniform: the uniqueness of the parallel, its multiplicity, its absence.
Four spaces are particular cases of the same two definitions: the flat plane, whose curvature is
zero; the hyperbolic plane, whose curvature is negative; and the two forms of the positive
curvature, the sphere and the elliptic plane \cite{riemann1854}. For the metric treatment of the
classical planes we refer to \cite{millman}.
A surface of \emph{variable} curvature is not obtained by gluing patches of those three planes:
along a seam the curvature would jump, and the lines would break. Such surfaces are the general
case, and the three planes the case of constant curvature \cite{riemann1854}. Smooth surfaces
whose curvature takes both signs do exist, and the classical example is the torus of revolution
in space, whose curvature is positive on the outer part, negative on the inner part, and zero on
the two circles that separate them \cite{docarmo1976}. On such a surface the two definitions keep
their wording, but the local curvature no longer decides the parallels: whether two lines meet is
a matter of the surface as a whole. The variable case belongs to the geometry of variable
curvature, and to the deformations that physics describes \cite{einstein1921}.
\end{remark}

\section{Conclusion}
\label{sec:conclusion}

The epigraph states the sense of the word primitive used here. Geometry does not begin with names
whose sense is fixed afterwards by axioms; it begins with a plane that one visualises, and the
names receive their sense from it. Euclid's postulate has to say that two lines, produced
indefinitely, meet, and its conclusion speaks of what happens arbitrarily far away. The rectitude
speaks of a quadrilateral that one draws, whose fourth angle one sees.

The rectitude is not an axiom under another name. An axiom of Hilbert's kind may be dropped, and
the words it contains given another sense: that is what its independence means, and the models
show that it can be done \cite{beltrami1868}. The rectitude cannot be dropped in that
way. If the fourth angle of a Lambert quadrilateral is acute or obtuse, the plane is no longer
the plane of ordinary intuition, and what we obtain is not that plane without an axiom but
another space, described on the first, as the models of the hyperbolic plane are drawn on the
rectilinear one.

In this sense geometry is the visualisation of our intuitive understanding of the universe.

\end{document}